\documentclass[12pt,reqno]{article}

\usepackage{amssymb}
\usepackage{amsmath}
\usepackage{amsfonts}
\usepackage{latexsym}
\usepackage{mathrsfs}
\usepackage{graphicx}
\usepackage{indentfirst,xspace,color}
\usepackage{pgf,tikz}
\usepackage{caption}
\usetikzlibrary{arrows}
\usepackage{hyperref}

\newcommand{\Dp}{\displaystyle}

\newcommand{\R}{{\mathbb R}}

\newtheorem{theorem}{Theorem}[section]
\newtheorem{lemma}{Lemma}[section]

\newtheorem{defi}{Definition}[section]
\newtheorem{remark}{Remark}[section]
\newenvironment{proof}{\noindent{\bf Proof:}}{\hfill$\Box$}

\newcommand{\N}{{\mathbb N}}

\def\XXint#1#2#3{{\setbox0=\hbox{$#1{#2#3}{\int}$}
\vcenter{\hbox{$#2#3$}}\kern-.5\wd0}}

\begin{document}
\title{A non-autonomous bifurcation problem for a non-local scalar one-dimensional parabolic equation
\thanks{
\newline {\small Instituto de Ci\^{e}ncias, Matem\'{a}ticas e Computa\c{c}\~{a}o/USP}}
}
\author{Alexandre N. Carvalho\thanks{Partially supported by FAPESP Grant  2018/10997-6 and CNPq Grant  303929/2015-4}, Yanan Li\thanks{Partially supported by NSFC Grant 11671367}, Tito L. M. Luna\thanks{Partially supported by FAPESP Grant  2019/20341-3 }\\  and Estefani M. Moreira\thanks{Partially supported by FAPESP Grant  2018/00065-9 }}


\maketitle

\begin{abstract}

In this paper we study the asymptotic behavior of solutions for a non-local non-autonomous scalar quasilinear parabolic problem in one space dimension. Our aim is to give a fairly complete description of the the forwards asymptotic behavior of solutions for models with Kirchoff type diffusion. In the autonomous we use the gradient structure of the model, some symmetry properties of solutions and develop comparison results to obtain a sequence of bifurcations of equilibria analogous to that seen in the model with local diffusivity. We give conditions so that the autonomous problem admits at most one positive equilibrium and analyse the existence of sign changing equilibria. Also using symmetry and our comparison results we construct what is called non-autonomous equilibria to describe part of the asymptotics of the associated non-autonomous non-local parabolic problem.

\end{abstract}

\section{Introduction and setting of the problem}

In this work we consider the following initial boundary value problem

\begin{equation}\label{1.1}
\left\{
  \begin{array}{ll}
    u_t-a(\|u_x\|^2)u_{xx}=\lambda u-\beta(t)u^3,  \ \ x\in (0, \pi), t>s,\\
    u(0, t)=u(\pi, t)=0,  \ \  t>s,\\
     u(x, s)=u_0(x),\ \ x\in (0, \pi),
  \end{array}
\right.
\end{equation}
where $\Dp\|u_x\|^2=\int^\pi_0|u_x(x)|^2dx$ (usual norm of the Hilbert space $H^1_0(0,\pi)$), $\lambda\in (0,\infty)$ is a parameter, $a:\R^+\to [1,2]$ is a locally Lipschitz function, $0<b_1<b_2$ and $\beta:\R^+\to [b_1,b_2]$ is a globally Lipschitz function.

\bigskip

The study of the inner structure of attractors for semilinear parabolic problems with local diffusivity has developed considerably and many very interesting results are available in the literatures (see, for example, \cite{Carvalho et al.} and references therein). The description of the inner structure for non-local models is much less exploited. Out aim is to provide some techniques to unravel the dynamics of such models in both autonomous and non-autonomous case.

\bigskip

Our problem has the origin in the developments that are concerned with the so called Chafee-Infante equation. The Chafee-Infante equation appeared in the literature for the first time in 1974, see \cite{CH-IN} and \cite{Chaf-Inf},
\begin{equation}\label{eq:C-I}
\left\{
  \begin{array}{ll}
    u_t-u_{xx}=\lambda u-bu^3,  \ \ x\in (0, \pi), t>s,\\
    u(0, t)=u(\pi, t)=0,  \ \  t>s,\\
     u(x, s)=u_0(x),\ \ x\in (0, \pi),
  \end{array}
\right.
\end{equation}
give us one of the few examples for which the structure of a global attractor is very well understood. We know that this model is gradient, that the number of equilibria, for $\lambda\in (n^2,(n+1)^2]$, is equal to $2n+1$, that the stable and unstable manifolds along a connection between two equilibria intersect transversaly (see {\color{red}\cite{HE,Angenent-SS})}), we know exactly the diagram o connections between two equilibria ({\color{red}see \cite{FR}}) and that it is structurally stable under autonomous and non-autonomous perturbations (see {\color{red}\cite{HE,Angenent-SS,BCL-Raugel})}).
For large non-autonomous perturbations in the parameter $b$ there has been some interesting development as well (see \cite{CLR-PAMS,Bcv}).

\bigskip

The introduction of a non-local diffusion changes everything. Most of the techiniques used in the constant diffusion case cannot be applied. Of course, the gradient structure is still present as well as some symmetry results, as we will see in the sequel. However the phase plane analysis used to construct the sequence of bifurcations of equilibria does not hold anymore as well as some important tools, like comparison and the Lap-Number, which will not hold automatically.

\bigskip

Let us comment, in some more detail, which important properties remain true for the non-local diffusion case and which do not. It is well known that the problem \eqref{eq:C-I} is globally well posed. Denote by $\{T(t):t\geqslant0\}$ the solution operator for it, that is, if $\R^+\ni t \mapsto u(t,u_0)\in H^1_0(0,\pi)$ denotes the global solution of \eqref{eq:C-I}, we write $T(t)u_0=u(t,u_0)$. The family $\{T(t):t\geqslant 0\}\subset C(H^1_0(0,\pi))$ is a semigroup. We can see that this semigroup is gradient, with Lyapunov function $V:H^1_0(0,\pi)\rightarrow\R$ given by
\begin{equation*}
V(u)=\frac{1}{2}\int^\pi_0\left( u_x^2(x) -{\lambda}u^2(x)+\frac{b}{2}u^4(x)\right)dx
\end{equation*} for all  $u\in H^1_0(0,\pi),$
that is,
	\begin{itemize}
	\item[i)] $V$ is continuous;
	\item[ii)] $[0,\infty)\ni t\mapsto V(T(t)u)\in\R$ is non-increasing;
	\item[iii)]  $V(T(t)u)=V(u)$ for all $t\geq 0$ implies that $u$ is a equilibrium for \eqref{eq:C-I}.
	\end{itemize}

Recall that, since the semigroup is gradient, the global attractor if it exists, is given by the set $\mathcal{A}=W^u(\mathcal{E}),$ where $\mathcal{E}$ is the set of equilibria and $W^u(\mathcal{E})$ its unstable manifold. In addition, if the number of equilibria is finite $\mathcal{E}=\{e_1, \cdots, e_n\}$ then 
\begin{equation}\label{eq:description}
W^u(\mathcal{E})=\cup_{i=1}^n W^u(e_i).
\end{equation} 
The later, which is known for the constant diffusion case, has to be proved for the non-local diffusion case.

\bigskip

In \cite{CH-IN}, the authors constructed a function that they called ``time map'' which provides information about the existence of the equilibria and it is related to the energy map
\[E(u,v):=\frac{v^2}{2}+\lambda\frac{u^2}{2} -\frac{bu^4}{4}.\]
 As the parameter $\lambda>0$ varies, we have the following sequence of bifurcations:

    \begin{theorem}\cite{CH-IN} \label{theo:bif-C-I}
Assume that $N^2<\lambda \leq (N+1)^2,$ for some $N\in \mathbb{N}^*.$ Then there exists a set of equilibria $\left\{\phi^\pm_{j,b}: 0\leq j\leq N\right\}$ for the equation \eqref{eq:C-I} 
such that:
        \begin{itemize}
            \item[i)] $\ \phi^+_{0,b}=\phi^-_{0,b}=0;$ 
            \item[ii)] For $1\leq j\leq N,$ $\phi^+_{j,b}$ has $j+1$ zeros in $[0,\pi]$ and they are given by $0,\frac{\pi}{j}, \frac{2\pi}{j},\dots, {\pi};$
            \item[iii)] $\phi^-_{j,b}=-\phi^+_{j,b},$ for all $1\leq j \leq N.$
            \item[iv)] There are no other equilibria for \eqref{eq:C-I}.
        \end{itemize}
    \end{theorem}

\bigskip

The phase plane analysis that led to the results of Theorem \ref{theo:bif-C-I} will not apply to the non-local diffusion as or to the non-autonomous problem. To overcome this difficultness we have pursued a different approach based on comparison results (which do not hold automatically, but we were able to prove) to construct subspaces of $H^1_0(0,\pi)$ which are positivelly invariant and used the gradient structure of the semigroup.

\bigskip

A similar approach has been used to study the inner structure of pullback attractors and uniform attractors (see \cite{Bcv,CLR-PAMS}) for a non-autonomous version of  \eqref{eq:C-I}. In order to describe the non-autonomous problem pursued in these works we will need to introduce some terminology.

\bigskip

Consider $(X,\Vert\cdot \Vert_X)$ a Banach space and denote $C(X)$ the space of continuous functions from $X$ into $X.$ If $\mathcal{P}=\{(t,s) \in \mathbb{R}\times\mathbb{R}: t\geq s\}$ we define
	\begin{defi}
	An evolution process $\{S(t,s) : (t,s)\in \mathcal{P}\} \subset C(X)$ is a family of maps that satisfies the following conditions:
		\begin{itemize}
		\item[i)] $S(t,t)=I_X,$ for all $t \in \mathbb{R},$ where $I_X$ denote the identity in $C(X);$
		\item[ii)] $S(t,s) S(s, \tau)=S(t,\tau),$ for all $t,s,\tau \in \mathbb{R}$ with $t\geq s\geq \tau;$
		\item[iii)] $\mathcal{P}\times X \ni (t,s,x)\longmapsto S(t,s)x$ is a continuous map.
		\end{itemize}
	\end{defi}
When $S(t,s)=S(t-s,0),$ for all $(t,s) \in \mathcal{P},$ the process will be called an autonomous process or semigroup and we will prefer the notation $\{T(t): t \geq 0\}$ where $T(t)=S(t,0).$

A global solution for the process $\{S(t,s): (t,s)\in \mathcal{P}\}$ is a function $\xi: \mathbb{R}\longrightarrow X$ such that \[S(t,s)\xi(s)=\xi(t), \quad \forall (t,s) \in \mathcal{P}.\]
If the set $\{\xi(t): t \in \R\}$ is bounded, $\xi$ is called a bounded solution.

\bigskip

Usually, we are interested in find out subsets of our space $X$ that provides information about the asymptotic behavior of the solutions and for that, we need the concept of a pullback attractor.

\begin{defi}
A family $\{A(t):t \in \mathbb{R}\}\!\subset\! X$ is the pullback attractor of $S(\cdot,\cdot)\subset C(X)$ if
\begin{itemize}
  \item[i)] $A(t)$ is a compact set, for each $t\in \mathbb{R};$
  \item[ii)] $S(t,s)A(s)=A(t)$ for all $t\geq s;$
  \item[iii)] The family $\{A(t):t \in \mathbb{R}\}$ pullback-attracts bounded sets of $X,$ that is, for each bounded $B \subset X,$ we have that \[\displaystyle\sup_{b \in B} d_X\left(S(t,s)b, A(t)\right)\longrightarrow 0 \mbox{ as } s\rightarrow -\infty;\]
    \item[iv)] $\{A(t):t \in \mathbb{R}\}$ is the minimal family of closed sets that satisfies the condition iii).
\end{itemize}
\end{defi}

If we assume that $\bigcup_{s\in \R} A(s)$ is a bounded set, then we have the following characterization for the pullback attractor: for all $t \in \R,$
\begin{equation}\label{eq:pa-charac}
A(t)=\left\{\xi(t):\ \xi:\R \longrightarrow X \mbox{ is a bounded global solution of } S(\cdot,\cdot) \right\}.
\end{equation}

\begin{defi}
A set $\mathbf{A}$ is the uniform attractor of $S(\cdot,\cdot)\subset C(X)$ if it is a compact subset of $X$ with the property that
$$
\sup_{\tau\in \R}\sup_{b\in B}d_X(S(t+\tau,\tau)b,\mathbf{A})\stackrel{t\to\infty}{\longrightarrow} 0
$$
for any $B\subset X$ bounded.
\end{defi}

For more details about process and pullback attractors, see \cite{Carvalho et al.}.

\bigskip

Consider the non-autonomous version of  \eqref{eq:C-I} given by
\begin{equation}\label{eq:carvalho}
\begin{cases}
u_t=u_{xx}+\lambda u- \beta(t)u^3, &\text{$t>0$, $x\in(0,\pi)$},\\
u(0,t)=u(\pi,t)=0, &\text{$t\geqslant 0$},\\
u(x,0)=u_0(x),&\text{$x\in(0,\pi)$ and $u_0\in H^1_0(0,\pi)$},
\end{cases}	
\end{equation}
where $\beta:\mathbb{R}\longrightarrow \mathbb{R}$ is a globally Lipschitz map and for all $t \in \mathbb{R},$ we have $b_1\leq \beta(t)\leq b_2,$ for some constants $b_1, b_2>0.$

\bigskip

In \cite{CLR-PAMS} the authors prove that if $N^2<\lambda \leqslant (N+1)^2$ there are $2N+1$ non-autonomous equilibria (including the zero equilibrium) for \eqref{eq:carvalho}, showing that the same sequence of bifurcation observed in the autonomous case, for \eqref{eq:C-I}, is also present in non-autonomous case. In the quest to find global bounded solutions that play the role of equilibria, in the description of attractors for autonomous gradient systems (as in \eqref{eq:description}), arise these non-autonomous equilibria (global bounded solutions which are non-degenerate at plus and minus infinity). In \cite{Bcv}, it is shown that the non-autonomous equilibria play such role (they are the key factor in the construction of the `lifted-invariant isolated invariant' sets to which the solutions converge and the rest of the uniform attractor correspond to `connections' between two of these, see \cite{Bcv} for more details.


%
%
%
%

\bigskip

One natural development was ask what happens with the structure of autonomous and non-autonomous attractors when the diffusion is non-local Kirchhof-type term $a$ as in \eqref{1.1}. We consider our equation \eqref{1.1} in the particular case when $\beta$ is time independent and when $\beta$ is time dependent.  That simple natural extension leads to the need of completely different techiniques than those used in \eqref{eq:C-I} where the authors use a phase plane analysis and a time map obtained to obtain the sequence of bifurcations depending on $\lambda$. An important aspect, that deserves to be pointed out, is that the operator $a\left(\|u_x\|^2\right)u_{xx}$ is non-linear and non-local, leading to several interesting and non-trivial mathematical questions about monotonicity (see \cite{BR,B,CA-GE}), comparison principles (c.f. \cite{CA-GE,AL-CO}), symmetry, (odd) extension and hyperbolicity among others.

\bigskip

To obtain the same sequence of bifurcations when $\beta$ is time dependent, the authors in \cite{CLR-PAMS} use the existence of a positive non-degenerate global solution (inspired in the work of \cite{RB-VL}) and invariant regions in $H_0^1(0,\pi)$ to construct the sequence of bifurcations. That technique can also be used to obtain, in a different way, the sequence of bifurcations of equilibria for \eqref{eq:C-I}. Its application to obtain the sequence of bifurcations  of equilibria for \eqref{eq:C-I} requires that the first eigenvalue of $-u_{xx}+\lambda u$ with Dirichlet boundary conditions be positive, that comparison results are available and the use of the gradient structure of the problem. Since \eqref{1.1} has a non-linear principal part, this analysis will require several adjustments and new ideas to be applied.

\bigskip

In section 2, we will develop our comparison result that will be essential to deal with the non-autonomous problem. In this section this will construct some bounded and positive invariant regions under the action of the process. In the section $3$, we will present our central result, Theorem \ref{theo:exist-pa}. In section 4, we sharpen the bifurcation result, using a new technique, obtaining that the sequence of bifurcations happen, in the autonomous case, at $a(0)N^2$, under some mild additional conditions on the diffusivity function $a$ and without using comparison. Finally, in section 5, we present some open questions for further investigation as well as the difficulties involved in each of these questions.




\section{Comparison results}
  Suppose that $u$ is a solution of \eqref{1.1}. For that, we are going to make a change in the variable $t$ and consider 
$\displaystyle\phi(t)=\int_{0}^{t}a(\Vert u_x(x,\sigma)\Vert^2)d\sigma$ and $w(x,\phi(t))=u(x,t)$ (see \cite{CH-VA-VC}), then 
 \begin{equation}\label{eq:nl_changed}
\left\{
  \begin{array}{ll}
      w_t = w_{xx} + \frac{\lambda w-\beta(\phi^{-1}(t))w^3}{a(\Vert w_x\Vert^2)}   \ \ x\in (0, \pi), t>s,\\
    w(0, t)=w(\pi, t)=0,  \ \  t>s,\\
     w(x, \phi(s))=u_0(x),\ \ x\in (0, \pi),
  \end{array}
\right.
\end{equation}

Problem \eqref{eq:nl_changed} is locally well-posed and the solutions are jointly continuous with respect time and initial conditions, see \cite{Carvalho et al.}.
By using the comparison result that we will develop, we can also prove that the solutions are globally defined. Hence, we can define the solution operators in the following way: if $u(t,s,u_0)$ is the solution of \eqref{eq:nl_changed} we write $S_\beta(t,s)u_0=u(t,s,u_0)$. That defines a solution operator family 
$\{S_\beta(t,s): t\geqslant s\}\subset C(H^1_0(0,\pi))$ associated to \eqref{eq:nl_changed}.

\bigskip

Note that $H^1_0(0,\pi)$ is an ordered Banach space with the following partial ordering structure:
\[ u\geq v \mbox{ in } H^1_0(0,\pi) \Leftrightarrow u(x)\geq v(x) \mbox{ a. e. for } x \in (0,\pi). \]

Using that order, we can define the positive cone as
\[\mathcal{C}=\{ u \in H^1_0(0,\pi) : \ u\geq 0 \}.\]

\bigskip

Consider the auxiliary initial boundary value problems
\begin{equation} \label{eq:b_1}
  \left\{
  \begin{array}{ll}
        z_t =  z_{xx} +\lambda z -\frac{b_1}{2} z^3  \\
    z(0, t)=z(\pi, t)=0,  \ \  t>0,\\
     z(\cdot, 0)=z_0(\cdot)\in H^1_0(0,\pi),
  \end{array}
\right.
\end{equation}
and
  \begin{equation} \label{eq:v/2}
  \left\{
  \begin{array}{ll}
       v_t = v_{xx} +\frac{\lambda }{2}v -b_2 v^3  \\
    v(0, t)=v(\pi, t)=0,  \ \  t>0,\\
     v(\cdot, 0)=v_0(\cdot)\in H^1_0(0,\pi).
  \end{array}
\right.
\end{equation}

\bigskip

It is well know that \eqref{eq:b_1} and \eqref{eq:v/2} are globally well posed and if $\R^+\ni t \mapsto z(t,z_0)\in H^1_0(0,\pi)$ and $\R^+\ni t \mapsto v(t,v_0)\in H^1_0(0,\pi)$ denote the solutions of \eqref{eq:b_1} and \eqref{eq:v/2}, we define the semigroups $\{T_1(t):t\geqslant 0\}$ and $\{T_2(t):t\geqslant 0\}$ by $T_1(t)z_0=z(t,z_0)$ and $T_2(t)v_0=v(t,v_0)$, $t\geqslant 0$, see \cite{Carvalho et al.}.

\bigskip

This section will be dedicated to the proof of the following result:

\begin{theorem} \label{theo:comparison}
With the above notation, if $u_0\leq u_1\leq u_2$ then
\begin{equation}\label{eq:comparison_la}
 T_2(t-s)u_0\leq S_\beta(t,s)u_1 \leq T_1(t-s) u_2, \quad \forall (t,s) \in \mathcal{P}.
\end{equation}
\end{theorem}

We will use the denomination ``positive solution'' for a global solution $\xi$ such that $\xi(t)\in \mathcal{C}$ for all $t \in \mathbb{R}$. If there exists a $\phi \in \mathcal{C}\cap\{\psi\in C^1(0,\pi):\psi'(0)\cdot\psi'(\pi)<0\} $ and $t_0 \in \mathbb{R}$ such that $\phi\leq \xi(t)$ for all $t\leq t_0$  (for all $t\geq t_0$) then $\xi$ will be called non-degenerate as $t\rightarrow -\infty$ 	(as $t\rightarrow +\infty$).

	\begin{defi}\label{def:na_equil}
	 A positive global solution $\xi$ of $S_\beta(\cdot, \cdot)$ is called a non-autonomous equilibrium if the  zeros of $\xi(t)$ are the same for all $t \in \mathbb{R}$ and $\xi$ is non-degenerate as $t\rightarrow \pm \infty.$
	\end{defi}

\bigskip

If we write \eqref{eq:nl_changed} in the abstract form
   \begin{equation}\label{eq:semilinear}
    \dot{u}+Au=f(t,u)
    \end{equation}
where
    \begin{itemize}


      \item[(1)]  $A:D(A)\subset L^2(0,\pi) \longrightarrow L^2(0,\pi)$ is the linear operator defined by $D(A)=H^2(0,\pi)\cap H^1_0(0,\pi)$ and $Au=-u_{xx}$, $u\in D(A)$. It is clear that $(0,\infty)\subset \rho(A)$ (resolvent of $A$) and that for all element $u_0 \in L^2(0,\pi)$ with $u_0\geq 0$ we have
\[(\lambda -A )^{-1}u_0 \geq 0, \quad \forall \lambda >0.\]
We express this fact by saying that $A$ has positive resolvent.

      \item[(2)] Consider $f:\R\times H^1_0(0,\pi) \longrightarrow L^2(0,\pi)$ a function such that for all $r>0$ we can find $\gamma(r)>0$ such that for all $t \in [t_0,t_1],$ the function $\gamma I_{L^2(0,\pi)} +f(t, u)$ is positive for all $u \in \mathcal{C}\cap B^{H^1_0(0,\pi)}_r(0).$

    \end{itemize}

Denote by $u_f(t,t_0,u_0)$ the solution of \eqref{eq:semilinear} for $t\geq t_0$ for which the solution is defined. The following theorem provides us a comparison result

\begin{theorem}\cite[Theorem 6.41]{Carvalho et al.}\label{theo:CLR6.41} If $A$ is as above and $f,g,$ and  $h$ are functions that satisfies (2). Then, we have the following
        \begin{itemize}
          \item[(i)] If for every $r>0$ there is a constant $\gamma=\gamma(r) >0$ such that $f(t,\cdot)+\gamma I$ is increasing in $B_{H^1_0(0,\pi)}(0,r),$ for all $t \in [t_0,t_1]$ and $u_0,u_1\in H^1_0(0,\pi)$ with $u_0\geq u_1,$ then $u_f(t,t_0,u_1)\geq u_f(t,t_0,u_0)$ as long as both solutions exist.
          \item[(ii)] If $f(t,\cdot)\geq g(t,\cdot)$ for all $t \in \R$ and $u_0\in H^1_0(0,\pi)$ then $u_f(t,t_0, u_0)\geq u_g(t,t_0,u_0)$ as long as both solutions exist.
          \item[(iii)] If $f,g $ are such that for every $r>0$ there exist a constant $\gamma=\gamma(r)>0$ and an increasing function $h(t,\cdot)$ such that, for every $t\in [t_0,t_1]$
            \[f(t,\cdot)+\gamma I \geq h(t,\cdot) \geq g(t,\cdot)+\gamma I\]
            in $B_{X^1}(0,r)$ and $u_0,u_1\in X^1$ with $u_0\geq u_1,$ then $u_f(t,t_0,u_0)\geq u_g(t,t_0,u_1)$ as long as both exist.
        \end{itemize}
    \end{theorem}

\bigskip
{\bf Proof of the theorem \eqref{theo:comparison}.} Observe that these equations are variations of \eqref{eq:C-I} and for $R>0$ we can find $\gamma(R)>0$ such that if $|u|\leq R$ and $t \in [s,t_0] $ then
\begin{equation}\label{eq:comparison}
0\leq \gamma u +\frac{\lambda }{2}u -b_2 u^3\leq \gamma u +\frac{\lambda u -\beta(\phi^{-1}(t))u^3}{a(\|u_x\|^2)}\leq \gamma u + \lambda u -\frac{ b_1}{2} u^3
\end{equation}
with $ \gamma u +\frac{\lambda u}{2} -b_2 u^3$ and $\gamma u + \lambda u  -\frac{ b_1}{2} u^3 $ increasing.

We now apply Theorem \ref{theo:CLR6.41} twice to compare solutions of \eqref{eq:nl_changed}, \eqref{eq:b_1} and \eqref{eq:v/2}. To that end, we define
\begin{equation}
\begin{split}
g_1(t,u)(x)&=\frac{\lambda u(x)-\beta(\phi^{-1}(t))u(x)^3}{a(\| u_x\|^2)}\quad
h_1(t,u)(x)=f_1(t,u)(x)=\lambda u(x)-\frac{b_1}{2}u(x)^3
\\
f_2(t,u)(x)&=\frac{\lambda u(x)-\beta(\phi^{-1}(t))u(x)^3}{a(\| u_x\|^2)}\quad
g_2(t,u)(x)=h_2(t,u)(x)=\frac{\lambda}{2} u(x)-b_2 u(x)^3
\end{split}
\end{equation}

Now, noticing that $H^1_0(0,\pi)$ is embedded in $L^\infty(0,\pi)$ and using \eqref{eq:comparison}, we are ready to apply Theorem \ref{theo:CLR6.41}, item iii) twice to compare \eqref{eq:nl_changed} with \eqref{eq:b_1} and \eqref{eq:v/2}.

\bigskip

In the next section we will use this comparison result to construct equilibria for \eqref{1.1}.

\section{The non-autonomous equation}

The process $\{S_\beta(t,s): t\geqslant s\}$ defined by \eqref{eq:nl_changed} admits a pullback attractor (see \cite{ACR-B}).

Observe that, by Theorem \ref{theo:bif-C-I}, if $\lambda >2,$ we can find a positive equilibrium $\phi_{1,b_1}^+$ for \eqref{eq:b_1} and a positive equilibrium $\phi_{1,b_2}^+$ for \eqref{eq:v/2}.

Using the Theorem \ref{theo:comparison} and the fact that $T_1(\cdot)$ is gradient, we have that
\[ \phi^+_{1,b_2}=T_2(t)\phi^+_{1,b_2}\leq T_1(t)\phi^+_{1,b_2}\stackrel{t\rightarrow +\infty}{\longrightarrow} \psi,\]
for some positive equilibrium $\psi$ of \eqref{eq:b_1}.

By the uniqueness of the positive equilibrium for \eqref{eq:b_1}, we conclude that $\psi =\phi^+_{1,b_1}.$

Therefore, $\phi^+_{1,b_2}\leq \phi^+_{1,b_1}$. 
Define the set 
\[X^+_1=\left\{
                \begin{array}{cl}
                 u \in H^1_0(0,\pi) \ : \ & \phi^+_{1,b_2}(x)\leqslant u(x) \leqslant \phi^+_{1,b_1}(x)\\
                                          & \mbox{ and }  u(x)=u(\pi-x) \mbox{ in } (0,\pi)
                 \end{array}
        \right\}.\]


\subsection{Construction of a positive non-autonomous equilibrium}

Our idea is to construct a positive non-autonomous equilibrium, see Definition \ref{def:na_equil}. For that, we will prove that $X^+_1$ is positively invariant, that is, $S(t,s)X^+_1\subset X^+_1,$ for all $(t,s) \in \mathcal{P}.$


For $x \in (0,\pi)$ and $u_0 \in X^+_1$ since $T_i(t-s)\phi^+_{1,b_i}=\phi^+_{1,b_i}, i=1,2,$ we have
\[\phi^+_{1,b_2}(x)\leqslant T_2(t-s)u_0\leqslant S(t,s)u_0\leqslant  T_1(t-s)u_0\leqslant \phi^+_{1,b_1}(x).\]

If $u(t,s, u_0)(x):= S(t,s)u_0(x),$ since $u_0(x)=u_0(\pi-x)$ for $x \in (0,\pi)$ both $u(t,s,u_0)(\cdot)$ and $u(t,s,u_0)(\pi - \cdot) \in H^1_0(0,\pi)$ with $t \in \R, t\geqslant s$ are solutions for the problem
\eqref{eq:nl_changed} then by uniqueness of solution we conclude that $u(t,s,u_0)(x)=u(t,s,u_0)(\pi-x),$ for all $x \in (0,\pi)$.

\begin{theorem} Suppose $\lambda > 2.$ Then the process $S_\beta(\cdot, \cdot)$ restricted to $X_1^+$ admits a pullback attractor. In particular, there exists a non-autonomous equilibrium in $\mathcal{C}.$
\end{theorem}

\begin{proof} The invariance follows from the reasoning that preceeded the theorem. The fact that $S_\beta(\cdot, \cdot)$ has a pullback attractor in $H^1_0(0,\pi)$ ensures that it also has a pullback attractor when restricted to $X^+_1$.
Now, any global solution in the pullback attractor of $S_\beta(\cdot, \cdot)$ restricted to $X^+_1$ is a non-autonomous equilibria.
\end{proof}

\begin{center}
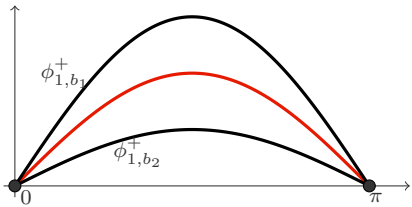

\begin{tikzpicture}[scale=1.5] \label{fig:positive_sol}
\definecolor{uuuuuu}{rgb}{0.2,0.2,0.2}
\definecolor{ffqqqq}{rgb}{0.9,0.1,0.0}
\draw[->,color=uuuuuu] (-0.1,0) -- (3.5,0.0);
\draw[->,color=uuuuuu] (0.0,-0.1) -- (0.0,1.6);
\draw[line width=1.2pt,color=ffqqqq,smooth,samples=100,domain=0.0:3.141592653589793] plot(\x,{sin(((\x))*180/pi)});
\draw[line width=1.2pt,smooth,samples=100,domain=0.0:3.141592653589793] plot(\x,{sin(((\x))*180/pi)/2.0});
\draw[line width=1.2pt,smooth,samples=100,domain=0.0:3.141592653589793] plot(\x,{1.5*sin(((\x))*180/pi)});
\begin{scriptsize}
\draw[color=uuuuuu] (1.0930274712527526,0.3) node {$\phi^+_{1,b_2}$};
\draw[color=uuuuuu] (0.45,1.0) node {$\phi^+_{1,b_1}$};
\draw [fill=uuuuuu] (0.0,0.0) circle (1.5pt);
\draw [fill=uuuuuu] (3.1415926501172664,0.0) circle (1.5pt);
\draw[color=uuuuuu] (3.2,-0.10334330006135964) node {$\pi$};
\draw[color=uuuuuu] (0.1,-0.10334330006135964) node {$0$};
\end{scriptsize}
\end{tikzpicture}
\captionof{figure}{ Region bounded by the positive equilibria $\phi^+_{1,b_1}$ and $\phi^+_{1,b_2}.$}
\end{center}

\subsection{Construction of other non-autonomous equilibria}

Suppose that $\lambda >2N^2,$ for $N \in \mathbb{N}^*.$  By Theorem \ref{theo:bif-C-I}, if $1\leq j\leq N,$ then there exists an equilibrium $\phi^+_{j,b_k}$ for $T_k(\cdot),$ $k=1,2,$ with $j+1$ zeros in $[0,\pi].$ For $1< j\leq N,$ consider the set $X^+_j=Y^+_j\cap Z_j,$ where
\[Y^+_j=\left\{u \in H^1_0(0,\pi): \min{\left(\phi^+_{j,b_1}(x),\phi^+_{j,b_2}(x)\right)}\leqslant u(x) \leqslant \max\left(\phi^+_{j,b_1}(x),\phi^+_{j,b_2}(x)\right)\right\}\]
and
\[Z_j = \left\{\begin{array}{cl} u \in H^1_0(0,\pi) : \ &u(x)=u\left(\tfrac{(2k-1)}{j}\pi-x\right) \mbox{ in } \left(0,\tfrac{2k-1}{j}\pi\right) \mbox{ and }\\ &u(x)=-u\left(\tfrac{2k}{j}\pi -x\right) \mbox{ in } (0,\tfrac{2k}{j}), \mbox{ where } k \in \N, 1\leqslant k\leqslant \tfrac{i}{2} \end{array} \right\}.\]

Let us prove that these sets are positively invariant.
We start with $j=2.$

\begin{center}
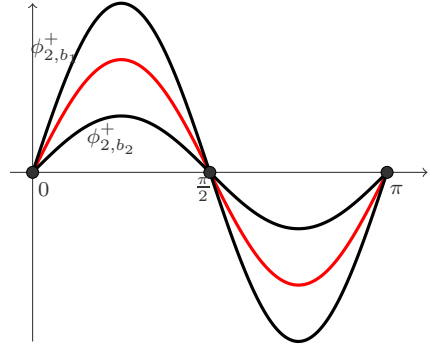

\begin{tikzpicture}[scale=1.5]
\definecolor{uuuuuu}{rgb}{0.2,0.2,0.2}
\definecolor{ffqqqq}{rgb}{1.0,0.0,0.0}
\draw[->,color=uuuuuu] (-0.2,0.0) -- (3.5,0.0);
\draw[->,color=uuuuuu] (0.0,-1.5) -- (0.0,1.5);
\draw[line width=1.2000000000000002pt,color=ffqqqq,smooth,samples=100,domain=0.0:3.141592653589793] plot(\x,{sin((2.0*(\x))*180/pi)});
\draw[line width=1.2000000000000002pt,smooth,samples=100,domain=0.0:3.141592653589793] plot(\x,{sin((2.0*(\x))*180/pi)/2.0});
\draw[line width=1.2000000000000002pt,smooth,samples=100,domain=0.0:3.141592653589793] plot(\x,{1.5*sin((2.0*(\x))*180/pi)});
\begin{scriptsize}
\draw[color=uuuuuu] (0.7,0.3) node {$\phi^+_{2,b_2}$};
\draw [fill=uuuuuu] (0.0,0.0) circle (1.5pt);
\draw [fill=uuuuuu] (3.1415926501172664,0.0) circle (1.5pt);
\draw[color=uuuuuu] (3.2267389268560303,-0.15) node {$\pi$};
\draw[color=uuuuuu] (0.2,1.1) node {$\phi^+_{2,b_1}$};
\draw [fill=uuuuuu] (1.570796326757987,0.0) circle (1.5pt);
\draw[color=uuuuuu] (1.5107381794925476,-0.15) node {$\frac{\pi}{2}$};
\draw[color=uuuuuu] (0.1,-0.15) node {$0$};
\end{scriptsize}
\end{tikzpicture}
\captionof{figure}{The set $X^+_2,$ the functions that lies between the solutions $\phi^+_{2,b_1}$ and $\phi^+_{2,b_2}.$ }
\end{center}

Consider $u_0 \in X^+_2$ then we know that $u_0 \in Z^+_2$ which means that $u_0(x)=-u_0(\pi-x).$ And by the uniqueness of solution, we have that $u(t,s, u_0)(x)=-u(t,s, u_0)(\pi-x),$ for all $x \in [0,\pi].$  With this we prove that $u \in Z^+_2.$ Moreover, we conclude that if $t\geqslant s$ then $u(t,s, u_0)(\tfrac{\pi}{2})=0$. Then we can use comparison to show that $u \in X^+_2:$
\[\begin{cases}
0 \leqslant T_2(t-s)u_0 \leqslant S(t,s)u_0\leqslant T_1(t,s)u_0 & \mbox{ in } [0,\tfrac{\pi}{2}]\\
T_1(t-s)u_0 \leqslant S(t,s)u_0\leqslant T_2(t,s)u_0 \leqslant 0 & \mbox{ in } [\tfrac{\pi}{2},\pi]
\end{cases}\]

Since $\phi^+_{2,b_i}$ is an equilibrium for $T_i(\cdot)$ and $0\leqslant \phi^+_{2,b_2}\leqslant u_0 \leqslant \phi^+_{2,b_1}$ in $[0,\tfrac{\pi}{2}]$ and \\ $\phi^+_{2,b_1}\leqslant u_0 \leqslant \phi^+_{2,b_2}\leqslant 0$ in $[\tfrac{\pi}{2},\pi],$ we can write
\[\begin{cases}
0 \leqslant \phi^+_{2,b_2} \leqslant S(t,s)u_0\leqslant \phi^+_{2,b_1} & \mbox{ in } [0,\tfrac{\pi}{2}]\\
\phi^+_{2,b_1} \leqslant S(t,s)u_0\leqslant \phi^+_{2,b_2} \leqslant 0 & \mbox{ in } [\tfrac{\pi}{2},\pi].
\end{cases}\]

Therefore, $X^+_2$ is positively invariant.

Before proving the case $X^+_3,$ we will prove the case $X^+_4.$ For that, just observe that if $\phi \in X^+_4$ then we have $\phi(x)=-\phi(\pi-x),$ for all $x \in (0,\pi)$ and, in particular, $u(t,s,\tfrac{\pi}{2})=0.$ Then, we can analyse the following problem
%
%
%
%
%
\begin{equation}\label{eq:X_4}
\begin{cases} u_t= u_{xx}+\frac{\lambda u -\beta(\phi^{-1}(t))u^3}{a(\Vert u_x\Vert^2)}\\
u(0,t)=u(\tfrac{\pi}{2},t)=0 & \mbox{ for all } t\geq s \\
u(x,\phi(s))={\phi}(x) & x \in [0,\tfrac{\pi}{2}]
\end{cases}
\end{equation}
Moreover, using the uniqueness of solution for \eqref{eq:X_4}, we conclude that
\[u(t,s,\phi)(x)=-u(t,s,\phi)(\tfrac{\pi}{2}-x), \ \mbox{ for } x \in [0,\tfrac{\pi}{2}]\]
and
\[u(t,s,\phi)(x)=-u(t,s,\phi)(\pi-x)\ \mbox{ for } x \in [0,\tfrac{\pi}{2}] (\mbox{ hence for } x\in [0,\pi]).\]

In particular, $u(t,s,\phi,\frac{\pi}{4})=0$ for all $t\geqslant s.$ and $u(t,s,\phi,\frac{3\pi}{4})=-u(t,s,\phi,\frac{\pi}{4})=0.$
With this, we can prove that $u$ lies in $Z^+_4.$ Now, we can use comparison to prove that $X^+_4$ is positively invariant.

To prove the invariance of $X^+_3,$ we define the following set
\[ W^+_4=\left\{\phi \in H^1_0\left(0,\tfrac{4\pi}{3}\right): \ \phi(x)=-\phi\left(\tfrac{4\pi}{3}-x\right), \mbox{ for } x \in \left[0,\tfrac{4\pi}{3}\right] \mbox{ and } \phi\mid_{[0,\pi]}\in X^+_3\right\}.\]

We have that $u(t,s,\phi, \tfrac{2\pi}{3})=0$ and can use the same idea as in $X^+_4$ to prove that $u(t,s,\phi, \tfrac{\pi}{3})=0.$ The comparison in $[0,\pi]$ follows similarly to the previous cases.

Therefore $X^+_3$ is invariant, since it is a restriction of $W^+_4.$

For the other cases, $\phi \in X^+_j,$ just observe that the invariance of $Z^+_j$ we can obtain using the reasoning applied in the previous cases and then we find that $u(t,s,\phi)(\frac{k\pi}{j})=0$ for $k=0,\dots, j.$
\[
    \begin{array}{rcccll}
     0\leq \ & T_2(t-s)\phi &\leq S(t,s)\phi &\leq T_1(t-s)\phi&        & \mbox{ in } [0,\tfrac{\pi}{j}] \\
	         & T_1(t-s)\phi &\leq S(t,s)\phi &\leq T_2(t-s)\phi& \leq 0 & \mbox{ in } [\tfrac{\pi}{j},\tfrac{2\pi}{j}] \\
	         & \vdots       & \vdots         & \vdots          &        &  \\
	         & T_1(t-s)\phi &\leq S(t,s)\phi &\leq T_2(t-s)\phi& \leq 0 & \mbox{ in } [\tfrac{(j-1)\pi}{j},\pi] \\
     0 \leq  & T_2(t-s)\phi &\leq S(t,s)\phi &\leq T_1(t-s)\phi&        & \mbox{ in } [\tfrac{(j-1)\pi}{j},\pi] \mbox{ and $j$ is odd}\\
	         & \Big( T_1(t-s)\phi &\leq S(t,s)\phi &\leq T_2(t-s)\phi& \leq 0 & \mbox{ in } [\tfrac{(j-1)\pi}{j},\pi] \mbox{ and $j$ is even}\Big)
    \end{array}
\]

Since $\phi \in Y^+_j,$ we can conclude that, for all $t\geq s,$
\[
\begin{array}{rcclll}
0\leq & \phi^+_{j,b_2}(x) &\leq S(t,s)\phi\leq & \phi^+_{j,b_1}(x) &       & \mbox{ in } [0,\tfrac{\pi}{j}] \\
	  & \phi^+_{j,b_1}(x) &\leq S(t,s)\phi\leq & \phi^+_{j,b_2}(x) &\leq 0 & \mbox{ in } [\tfrac{\pi}{j},\tfrac{2\pi}{j}] \\
      &\vdots             & \vdots             & \vdots            &       & \\
	  & \phi^+_{j,b_1}(x) &\leq S(t,s)\phi\leq & \phi^+_{j,b_2}(x) &\leq 0 & \mbox{ in } [\tfrac{(j-1)\pi}{j},\pi] \mbox{ and $j$ is even} \\
\Big(0\leq &\phi^+_{j,b_2}(x)  &\leq S(t,s)\phi\leq & \phi^+_{j,b_1}(x) &       & \mbox{ in } [\tfrac{(j-1)\pi}{j},\pi] \mbox{ and $j$ is odd}\Big)
\end{array}
\]

With this, we conclude that $S(t,s)\phi \in X^+_j$ for all $t\geqslant s.$ Therefore, we prove that $X^+_j$ is positively invariant and then $S(\cdot, \cdot)$ restricted to $X^+_j$ admits a pullback attractor. We can summarize the results in the following

\begin{theorem}\label{theo:exist-pa}
 Suppose that, for some $N \in \mathbb{N},$ we have $2N^2< \lambda\leq 2(N+1)^2.$ Then, for $j=1,\dots, N,$ the process $S_\beta(\cdot, \cdot)$ restricted to $X^+_j$ admits a pullback attractor.

In particular, for each $j=1,\dots, N,$ there exists a non-autonomous equilibrium $\xi^+_j$ that has $j+1$ zeros in $[0,\pi].$
\end{theorem}

\begin{remark}
Note that if $\lambda >2N^2,$  for $1 \leq j\leq N,$ there exists an equilibrium $\phi^-_{j,b_k}$ for $T_k(\cdot),$ $k=1,2,$ with $j+1$ zeros in $[0,\pi].$ Then, we can define the set $X^-_j=Y^-_j\cap Z_j,$ where
\[Y^-_j=\left\{u \in H^1_0(0,\pi): \min{\left(\phi^-_{j,b_1}(x),\phi^-_{j,b_2}(x)\right)}\leq u(x) \leq \max\left(\phi^-_{j,b_1}(x),\phi^-_{j,b_2}(x)\right)\right\}.\]

We can also prove that $S(t,s)X_j^-\subset X_j^-,$ for all $t\geq s.$
\end{remark}

Observe that all the construction were developed for solutions of \eqref{eq:nl_changed}. But we just need to remember that the change of variables only affects $t,$ hence we constructed a set of bounded non-autonomous equilibria also for \eqref{1.1}. We can summarize the result in the following

\begin{theorem}
Suppose that $\lambda > 2N^2,$ for $N \in \mathbb{N}^*.$ The problem \eqref{1.1} has at least $2N$ non-autonomous equilibria.
\end{theorem}

\section{The autonomous problem}

Consider the autonomous problem
	\begin{equation}\label{eq:bproblem}
	 \begin{cases}
u_t=a(\|u_x\|^2)u_{xx}+\lambda u- bu^3, &\text{$t>0$, $x\in(0,\pi)$},\\
	u(0,t)=u(\pi,t)=0, &\text{$t\geqslant 0$},\\
	u(0)=u_0\in H^1_0(0,\pi).
\end{cases}	
	\end{equation}
for some constant $b>0.$

The map $V:H^1_0(0,\pi)\longrightarrow \R$ is a Lyapunov function for \eqref{eq:bproblem}
   \[
V(u)=\frac{1}{2}\int^{\|u_x\|_j^2}_0a(s)ds+\int^\frac{\pi}{j}_0\left(-\frac{\lambda}{2}(u(x))^2+\frac{b}{4}(u(x))^4\right)dx.
\]

Hence, the semigroup   $\{T_b(t): t\geq 0\}$ associated to \eqref{eq:bproblem} is gradient.

Observe also that we can use the analysis from the previous section, taking $b_1=b_2=b$ in \eqref{eq:b_1} and \eqref{eq:v/2} and we have the comparison \eqref{eq:comparison_la}.

Suppose that $\lambda >2.$ Then, if we take some $u_0 \in X^+_1,$ we have that
\[\phi^+_{1,b_2}\leq T_b(t)u_0\stackrel{t\rightarrow +\infty}{ \longrightarrow } \psi\]
where $\psi$ is an equilibrium of the equation in \eqref{eq:bproblem}. Observe that $\psi\in X_1^+.$

Therefore, with this, we proved that if $\lambda >2,$ we have an positive equilibrium for \eqref{eq:bproblem}. In the Appendix, section \ref{sec:posit-equil}, we show that if $a$ is a non-decreasing function then we can find only one positive and symmetric equilibrium when $\lambda >a(0).$

Observe that we could use the same arguments and the fact that $T_b(\cdot)$ is gradient to guarantee that if $\lambda>2N^2,$ for some $N\in \mathbb{N}^*,$ then we can construct $2N$, $\{\phi_j^\pm: 1\leq j\leq N\}$, non-zero equilibria for \eqref{eq:bproblem} with the properties $ii)$ and $iii)$ of Theorem \ref{theo:bif-C-I} and such that $\phi_j(x)=\phi_j^+(\pi/j-x)$, $x\in (0,\pi/j)$, and $\phi_j(x)=-\phi_{j}(\pi/j+x)$, $x\in (0,\pi-\pi/j))$, $1\leq j\leq N$.

Note that, the bifurcations are expected to happen at $a(0)N^2$. In fact, when $a$ is a non-decreasing function, the bifurcations of equilibria happen every time $\lambda$ passes $a(0)N^2, $ for $N \in \mathbb{N}^*$. The proof of that fact can be found in Section \ref{section:Equilibria}.

\subsection{An auxiliary elliptic problem}\label{sec:posit-equil}
We want to show another approach to the autonomous problem \eqref{eq:bproblem}, when we study the existence of a positive, nontrivial and symmetric equilibria by minimizing the energy of the elliptical problem
    \begin{equation}\label{auxiliaryproblem}
        \begin{cases}
        -a\left(\|u_x\|_j^2\right)u_{xx}=\lambda u- b u^3, & x\in(0,\frac{\pi}{j}),\\
	       u(0)=u(\frac{\pi}{j})=0 &\\
	   \end{cases}
    \end{equation}
where $j\in \N $ ($j=1,2,\cdots$), $\Dp\|u_x\|_j^2=\int^\frac{\pi}{j}_0|u_x(x)|^2dx$ (usual norm of the Banach space $H^1_0(0,\frac{\pi}{j})$), constants $\lambda>0$,  $b>0$ and $a:\R\rightarrow\R$ a continuous function such that
\begin{equation}\label{a}
0<\sigma\leqslant a(t), \qquad\forall t\in \R.
\end{equation}
holds.

We say that $u\in H^1_0(0,\frac{\pi}{j})$ is a weak solution of the problem \eqref{auxiliaryproblem} if
		\begin{equation}\label{solutionauxialiaryproblem}
        	\Dp a\left(\|u_x\|_j^2\right)\int^\frac{\pi}{j}_0 u_x(x)v_x(x)dx+\int^\frac{\pi}{j}_0\left(-\lambda u(x)+b(u(x))^3\right)v(x)dx=0,
		\end{equation}for all $v\in H^1_0(0,\frac{\pi}{j})$.\\
		
Before we start to analyse this problem, remember that the operator\\ $L:H^2(0,\frac{\pi}{j})\cap H^1_0(0,\frac{\pi}{j})\longrightarrow H^1_0(0,\frac{\pi}{j})$ given by $Lu=-\Delta u$ is self-adjoint with a discrete set as spectrum (only the point spectrum is non-empty ) given by $\{j^2N^2 : N \in \mathbb{N}\}.$ In particular, its first eigenvalue is the number $\lambda_1=j^2.$

Note that a solution of  \eqref{solutionauxialiaryproblem} can be founded as critical point of the energy functional  $E:H^1_0(0,\frac{\pi}{j})\rightarrow\R$ defined by
    \begin{equation*}
        E(u)=\frac{1}{2}\int^{\|u_x\|_j^2}_0a(s)ds+\int^\frac{\pi}{j}_0\left(-\frac{\lambda}{2}(u(x))^2+\frac{b}{4}(u(x))^4\right)dx
    \end{equation*}
for all  $u\in H^1_0(0,\frac{\pi}{j})$. Note that
	\begin{equation*}
    	\langle E'(u),v\rangle=a\left(\|u_x\|^2\right)\int^\frac{\pi}{j}_0 u_x(x)v_x(x)dx+\int^\frac{\pi}{j}_0\left(-\lambda u(x)+b(u(x))^3\right)v(x)dx.
	\end{equation*} for all $u,v\in H^1_0(0,\frac{\pi}{j})$.\\
		
		First, we can prove that $E$ is a coercive functional on $H^1_0(0,\frac{\pi}{j})$ as a consequence of the inequality $\Dp E(u)\geqslant\frac{\sigma}{2}\|u_x\|_j^2-\frac{\lambda^2\pi}{4bj}$. Also, $E$ is weakly lower semicontinuous on $H^1_0(0,\frac{\pi}{j})$. Thus,
		
	\begin{lemma}\label{lemaxistenceauxiliaryproblem}
		Assume that the function $t\rightarrow a(t)$ is only continuous and \eqref{a}
 is valid, for some constant $\sigma >0$.
 If $\lambda> a(0)j^2,$ for $j \in \mathbb{N}^*$ there exists a positive, nontrivial and symmetric weak solution  for the problem \eqref{auxiliaryproblem}.
		\end{lemma}
	\begin{proof}
		Since we are interested in positive and symmetric weak solutions for the problem \eqref{auxiliaryproblem}, we restrict the domain of $E$ to
    \begin{equation*}
    \mathcal{M}=\left\{v\in H^1_0(0,\frac{\pi}{j});\text{$v(x)=v(\frac{\pi}{j}-x)$ and $0\leqslant v \leqslant\sqrt{\frac{\lambda}{b}}$  almost everywhere in $(0,\frac{\pi}{j})$}\right\}.
    \end{equation*}
Clearly $\mathcal{M}$ is weakly closed, that is $\mathcal{M}$ and $E$ satisfies all the conditions of the theorem 1.2 of \cite{ST}. Then we infer the existence of the relative minimizers $u\in \mathcal{M}$. To show that $u$ is a weak solution of \eqref{auxiliaryproblem}, we consider $\varphi\in C^\infty_c(0,\frac{\pi}{j})$ and $\epsilon>0$. Let $v_\epsilon=u+\epsilon\varphi-\varphi^\epsilon+\varphi_\epsilon\in \mathcal{M}$, where
		\begin{eqnarray*}
	       \varphi^\epsilon=\max\left\{0,u+\epsilon\varphi-\sqrt{\frac{\lambda}{b}}\right\}\geqslant0&\quad \text{and}\quad
		   \varphi_\epsilon=\max\{0,-(u+\epsilon\varphi)\}\geqslant0,
	\end{eqnarray*}
we note that $\varphi^\epsilon,\varphi_\epsilon\in H^1_0(0,\frac{\pi}{j})\cap L^\infty(0,\frac{\pi}{j})$.\\
	Now we have the following estimates
	\begin{eqnarray*}
	\langle E'(u),\varphi^\epsilon\rangle&=& a(\|u_x\|_j^2)\int^\frac{\pi}{j}_0u_x(x)\varphi^\epsilon_x(x)+\int^\frac{\pi}{j}_0\left(-\lambda u(x)+b(u(x))^3\right)\varphi^\epsilon(x)dx\\
	&=&a(\|u_x\|_j^2)\int_{\Omega^\epsilon}u_x(x)(u_x+\epsilon\varphi_x)(x)+\int_{\Omega^\epsilon} \left(-\lambda u(x)+b(u(x))^3\right)(u+\epsilon\varphi-\sqrt{\frac{\lambda}{b}})(x)dx\\
	&\geqslant& a(\|u_x\|_j^2)\int_{\Omega^\epsilon}\epsilon\varphi_x(x)+\int_{\Omega^\epsilon} \left(-\lambda u(x)+b(u(x))^3\right)(u+\epsilon\varphi-\sqrt{\frac{\lambda}{b}})(x)dx\\
	&\geqslant& a(\|u_x\|_j^2)\int_{\Omega^\epsilon}\epsilon\varphi_x(x)+\int_{\Omega^\epsilon} \left(-\lambda u(x)+b(u(x))^3\right)\epsilon\varphi(x)dx\\
	&\geqslant& -a(\|u_x\|_j^2)|\Omega^\epsilon|\epsilon\|\varphi_x\|_{L^\infty(0,\frac{\pi}{j})}-\frac{2\lambda}{3}\left(\frac{\lambda}{3b}\right)^{\frac{1}{2}}|\Omega^\epsilon|\epsilon\|\varphi\|_{L^\infty(0,\frac{\pi}{j})}\\
	&\geqslant& \left[-a(\|u_x\|_j^2)\|\varphi_x\|_{L^\infty(0,\frac{\pi}{j})}-\frac{2\lambda}{3}\left(\frac{\lambda}{3b}\right)^{\frac{1}{2}}\|\varphi\|_{L^\infty(0,\frac{\pi}{j})}\right]\epsilon|\Omega^\epsilon|\\
	\end{eqnarray*}
	where $\Dp\Omega^\epsilon:=\left\{x\in(0,\frac{\pi}{j});u(x)+\epsilon\varphi(x)\geqslant\sqrt{\frac{\lambda}{b}}>u(x)\right\}$, that satisfies $|\Omega^\epsilon|\rightarrow0$ when $\epsilon\rightarrow0$. Similarly	

\begin{equation*}
\begin{split}
\langle E'(u),\varphi_\epsilon\rangle &= a(\|u_x\|_j^2)\int^\frac{\pi}{j}_0u_x(x)(\varphi_\epsilon)_x(x)+\int^\frac{\pi}{j}_0\left(-\lambda u(x)+b(u(x))^3\right)\varphi_\epsilon(x)dx\\
&= -a(\|u_x\|_j^2)\int_{\Omega_\epsilon}\!u_x(x)(u_x+\epsilon\varphi_x)(x)\!-\!\int_{\Omega_\epsilon}\! \left(-\lambda u(x)+b(u(x))^3\right)\!(u+\epsilon\varphi)(x)dx\\
&\leqslant -a(\|u_x\|_j^2)\int_{\Omega_\epsilon}\epsilon\varphi_x(x)-\int_{\Omega_\epsilon} (-\lambda u(x)+b(u(x))^3)\epsilon\varphi(x)dx\\
&\leqslant a(\|u_x\|_j^2)|\Omega_\epsilon|\epsilon\|\varphi_x\|_{L^\infty(0,\frac{\pi}{j})}+\frac{2\lambda}{3}\left(\frac{\lambda}{3b}\right)^{\frac{1}{2}}|\Omega_\epsilon|\epsilon\|\varphi\|_{L^\infty(0,\frac{\pi}{j})}\\
&\leqslant \left[a(\|u_x\|_j^2)\|\varphi_x\|_{L^\infty(0,\frac{\pi}{j})}+\frac{2\lambda}{3}\left(\frac{\lambda}{3b}\right)^{\frac{1}{2}}\|\varphi\|_{L^\infty(0,\frac{\pi}{j})}\right]\epsilon|\Omega_\epsilon|
\end{split}
\end{equation*}
where, $\Dp\Omega_\epsilon:=\left\{x\in(0,\frac{\pi}{j});u(x)+\epsilon\varphi(x)\leqslant0<u(x)\right\}$ that satisfies $|\Omega^\epsilon|\rightarrow0$ when $\epsilon\rightarrow0$.

Since that, $E$ is differentiable in $u$ and $E(u)\leqslant E(v_\epsilon)$ we have
\begin{equation*}
\langle E'(u),\varphi\rangle\geqslant\frac{\langle E'(u),\varphi^\epsilon\rangle-\langle E'(u),\varphi_\epsilon\rangle}{\epsilon}\geqslant C_1|\Omega_\epsilon|-C_2|\Omega^\epsilon|
\end{equation*}
So for $\epsilon\rightarrow0$ we obtain $\langle E'(u),\varphi\rangle\geqslant0$ for each $\varphi\in C^\infty_c(0,\frac{\pi}{j})$. By reversing of the signal and $\varphi\in \overline{C^\infty_c(0,\frac{\pi}{j})}=H^1_0(0,\frac{\pi}{j})$ we get $ E'(u)=0$.
		
Finally, we will to show that $u$ is a non-trivial. For that, we will show that the minimum of energy is  negative what it guarantees that $u$ could not be zero. In fact, let $\phi$ be the eigenfunction associated to the first eigenvalue $j^2$ of the operator $u_{xx}$ on $H^1_0(0,\frac{\pi}{j})$, that is, $\phi$ is  solution to the following eigenvalue problem

\begin{equation*} 
\begin{split}
-&\phi_{xx}=j^2 \phi,  x\in(0,\frac{\pi}{j}),\\
&\phi(0)=\phi(\frac{\pi}{j})=0. \\
\end{split}
\end{equation*}
Since that $a(0)j^2<\lambda$, from continuity of the function $a$, we have
$a(t)\lambda_1<\lambda$ for each $\Dp t\in\left[0,\lambda_1\delta^2\int^\frac{\pi}{j}_0(\phi(x))^2dx\right]$ for some $\delta>0$ small enough. Note that $\delta\phi\in\mathcal{M}$, thus

\begin{equation*}
\begin{split}
E(\delta\phi)&=\frac{1}{2}\int^{\|\delta\phi_x\|_j^2}_0a(s)ds+\int^\frac{\pi}{j}_0\left(\frac{b}{4}(\delta\phi(x))^4-\frac{\lambda}{2}(\delta\phi(x))^2\right)dx\\
&=\frac{1}{2}\delta^2\left[[a(c_\delta)j^2-\lambda]\int^\frac{\pi}{j}_0(\phi(x))^2dx+\frac{b}{2}\delta^2\int^\frac{\pi}{j}_0(\phi(x))^4dx\right],\\
&\text{for some $c_\delta\in\left[0,j^2\delta^2\int^\frac{\pi}{j}_0(\phi(x))^2dx\right]$,}\\
&<0.
\end{split}
\end{equation*}

\end{proof}
	\begin{lemma}\label{lemmauniqueness}If  a function $t\rightarrow a(t)$ is only continuous and \emph{non-decreasing} and \eqref{a} holds then a positive and non-trivial solution of the problem \eqref{auxiliaryproblem} is unique.
    \end{lemma}
	
	\begin{proof} Assume, by contradiction, that $u$ and $v$ are two distinct nontrivial non-negative solutions of \eqref{auxiliaryproblem}. By regularity we have $u, v \in C^1(0,\frac{\pi}{j})$ and by the maximum principle we have $u, v > 0$ in $(0,\frac{\pi}{j})$ (see \cite{PR-WE}), so that $\Dp\frac{u^2}{v},\frac{v^2}{u}\in H^1_0(0,\frac{\pi}{j})$. Thus,
\begin{eqnarray*}
	0\!&\!\leqslant\!&\! \left(a(\|u_x\|_j^2)-a(\|v_x\|_j^2)\right)\left(\|u_x\|_j^2-\|v_x\|_j^2\right)+a(\|v_x\|_j^2)\int^{\frac{\pi}{j}}_0\left(u_x(x)-\frac{u(x)}{v(x)}v_x(x)\right)^2dx+\\
	& & a(\|u_x\|_j^2)\int^{\frac{\pi}{j}}_0\left(v_x(x)-\frac{v(x)}{u(x)}u_x(x)\right)^2dx\\
	&=& a(\|u_x\|_j^2)\int^{\frac{\pi}{j}}_0u_x^2(x)dx-a(\|u_x\|_j^2)\int^{\frac{\pi}{j}}_0u_x(x)\left(\frac{v^2}{u}\right)_x(x)dx+a(\|v_x\|_j^2)\int^{\frac{\pi}{j}}_0v_x^2(x)dx\\
	&&-a(\|v_x\|_j^2)\int^{\frac{\pi}{j}}_0v_x(x)\left(\frac{u^2}{v}\right)_x(x)dx\\
	&=& b\int^{\frac{\pi}{j}}_0\left(v^2(x)-v^2(x)\right)\left(u^2(x)-v^2(x)\right)dx.
			\end{eqnarray*} If $u\not\equiv v$, the last integral is negative and we obtain a contradiction.
	\end{proof}\\

\subsection{Equilibria for the autonomous problem}	\label{section:Equilibria}
		
In this section, we are also assuming that $a$ satisfies \eqref{a} and also that $a$ is non-decreasing. Depends on the position of parameter $\lambda >0$, we can inductively construct equilibria the change sign for the problem \eqref{eq:bproblem}.

 Notice that finding equilibria for the equation\eqref{eq:bproblem} is finding the solutions for
	\begin{equation}\label{eq:aut_pi}
	 \begin{cases}
	 	a(\Vert u_x\Vert^2 )u_{xx}+\lambda u-b u^3=0 \qquad \text{in $(0,\pi)$} \\
	 	u(0)=u({\pi})=0
	 \end{cases}.	
	\end{equation}

First, if $\lambda<a(0)$ then  $0$ is the only solution of the problem \eqref{eq:aut_pi}.

In the case, $a(0)< \lambda,$ 
by Lemma \ref{lemaxistenceauxiliaryproblem} and  Lemma \ref{lemmauniqueness}, we can still find a unique positive symmetric weak solution for the problem \eqref{eq:aut_pi}, that we will denote $\phi_{1,b}^+.$ Observe that $\phi_{1,b}^-=-\phi_{1,b}^+$ is also an equilibrium for \eqref{eq:aut_pi}.

In the case $a(0)2^2< \lambda,$ 
we have the equilibria $0$, $\phi_{1,b}^+$ and $\phi_{1,b}^-$ and we can also construct a pair of equilibria that change sign one time. For that, we are going to restrict ourselves to the following problem in $[0,\frac{\pi}{2}]:$
	\begin{equation}\label{eq:aut_pi/2}
	 \begin{cases}
	 	a(\Vert u_x\Vert^2 )u_{xx}+\lambda u-\frac{b}{2}u^3=0  \\
	 	u(0)=u(\frac{\pi}{2})=0
	 \end{cases}	
	\end{equation}

Since, in this case, we have that $\lambda > a(0)2^2,$ the problem \eqref{eq:aut_pi/2} has a positive solution $\phi_{1,\frac{\pi}{2}}$ with $\phi_{1,\frac{\pi}{2}}(0)=\phi_{1,\frac{\pi}{2}}(\frac{\pi}{2})=0$ and that satisfies $u(x)=u(\frac{\pi}{2}-x)$ for all $x \in (0,\frac{\pi}{2}).$

We define $\phi^+_{2,b}(x)=\Dp\begin{cases}\frac{1}{\sqrt{2}}\phi_{1,\frac{\pi}{2}}(x), \mbox{ if } x\in [0,\frac{\pi}{2}] \\
-\frac{1}{\sqrt{2}}\phi_{1,\frac{\pi}{2}}(\pi-x), \mbox{ if } x\in [0,\frac{\pi}{2}] \end{cases}$

Notice that $\phi^+_{2,b}$ is a solution of \eqref{eq:aut_pi}. Also, $\phi^+_{2,b}(\pi-x)=-\phi^+_{2,b}(x),$ for all $x \in (0,\pi):$ if  $x \in [0,\frac{\pi}{2}],$ we have that
\[\phi^+_{2,b}(\pi-x)=-\tfrac{1}{\sqrt{2}}\phi_{1,\frac{\pi}{2}}(\pi-(\pi-x))=-\tfrac{1}{\sqrt{2}}\phi_{1,\frac{\pi}{2}}(x))=\phi^+_{2,b}(x)\] and if $x \in [\frac{\pi}{2}, \pi],$ then
\[\phi^+_{2,b}(\pi-x)=\tfrac{1}{\sqrt{2}}\phi_{1,\frac{\pi}{2}}(\pi-x)=-(-\tfrac{1}{\sqrt{2}}\phi_{1,\frac{\pi}{2}}(\pi-x))=\phi^+_{2,b}(x).\]

\[Y^+_2 := \left\{u \in H^1_0(0,\pi) : u(x)=-u(\pi-x) \mbox{ in } [0,\pi] \mbox{ and } u(x)=u(\pi/2 -x) \mbox{ in } [0,{\pi}/{2}]\right\}.\]	
	
An inductive argument can be applied to it and we can show that $a(0)j^2$ is a bifurcation point of the parameter $\lambda>0.$ We summarize the result in the following
	
    \begin{theorem}
    If $a(0)N^2< \lambda\leqslant a(0)(N+1)^2,$ then there are $2N+1$ equilibria of the equation \eqref{eq:bproblem}; $\{0\}\cup\{\phi_{j,b}^\pm: j=1,\dots, N\}$,
where $\phi^+_{j,b}$ and $\phi^-_{j,b}$  have $j+1$ zeros in $[0,\pi]$ and $\phi^-_{j,b}=-\phi^+_{j,b}$.
    \end{theorem}

%

We conclude that  the only equilibria for the problem \eqref{eq:bproblem} since the function $a$ is non-decreasing.

\section{Conclusion}

We constructed non-autonomous equilibria for \eqref{1.1} depends on the parameter $\lambda >0.$	This was interesting step, and it was the first one, to try to describe the structure of the attractor. There are still several interesting open questions that remain. One of them is that: Can we find other non-autonomous equilibria besides the ones we constructed?
Our expectations is that the
answer is no.

We still want to understand where the bifurcations happen. We believe that always that $\lambda >a(0)N^2,$ we will have a bifurcation and we can construct two symmetric solutions that change signs $N+1$ times in $[0,\pi].$

We can also investigate the equilibria we constructed, see if we have results of stability and hyperbolicity, see if we can describe the connections between them.

For the autonomous Chafee-Infante \eqref{eq:C-I}, we know that all the equilibria are hyperbolic and that the positive one is stable. We know the connections between the equilibria, 0 is connected to the other equilibria and we have connections between one equilibria to another one with less zeros. The proof of this results require the lap-number property and the fact that the operator is a Liouville-Sturm operator.

In the case \eqref{eq:carvalho}, the authors could prove the results related to stability of the equilibria and also proved a similar connection between the equilibria, that is because here, we can still use the lap-number property. But, the hyperbolicity is an open question. Hyperbolicity is a challenging subject in the non-autonomous case.

In \eqref{eq:carvalho}, they also show that the uniform attractor can be described using the non-autonomous equilibria for problems involving limit of translations of the function $\beta(\cdot)$ (functions in the global attractor of the driving semigroup). In \cite{CH-MA} the authors prove that the $\omega-$limit sets of solutions consist of symmetric functions. To that end they extend our Dirichlet problem to a periodic problem in an interval twice as big and show that, using the lap-number, we can guarantee that the $\omega-$limit is a union of ``sets of symmetric solutions''.

We believe that, if one can prove the lap-number-like properties for solutions, one will be able to give a full characterization of the uniform attractor.

(Alexandre N. Carvalho) Departamento de Matem\'{a}tica, Instituto de Ci\^{e}ncias Ma\-te\-m\'{a}\-ti\-cas e de Computa\c{c}\~{a}o, Universidade de S\~{a}o Paulo-Campus de S\~{a}o Carlos\\
Email addres: {\tt andcarva@icmc.usp.br}\\

(Tito L.M. Luna) Departamento de Matem\'{a}tica, Instituto de Ci\^{e}ncias Ma\-te\-m\'{a}\-ti\-cas e de Computa\c{c}\~{a}o, Universidade de S\~{a}o Paulo-Campus de S\~{a}o Carlos\\
Email addres: {\tt titoluna@dm.ufscar.br}\\

(Estefani M. Moreira) Departamento de Matem\'{a}tica, Instituto de Ci\^{e}ncias Ma\-te\-m\'{a}\-ti\-cas e de Computa\c{c}\~{a}o, Universidade de S\~{a}o Paulo-Campus de S\~{a}o Carlos\\
Email addres: {\tt estefani@usp.br}\\

(Yanan Li) School of Mathematics and Statistics, Zhengzhou
University, China\\
Email addres: {\tt lyn20112110109@163.com}

	\end{document}